\documentclass[psamsfonts]{amsart}

\usepackage{amssymb,amsfonts}
\usepackage[all,arc]{xy}
\usepackage{enumerate}
\usepackage{mathrsfs}
\usepackage{hyperref}
\usepackage{graphicx}
\usepackage{booktabs}
\usepackage[dvipsnames,svgnames]{xcolor}
\usepackage{marginnote}
\usepackage{soul}
\usepackage{float}
\graphicspath{ {./Images/} }

\allowdisplaybreaks 

\usepackage{tikz-cd}
\newcommand{\C}{\mathbb C}

\newcommand{\F}{\mathbb F}

\newcommand{\Q}{\mathbb Q}
\newcommand{\Z}{\mathbb Z}

\newcommand{\fN}{\mathfrak N}

\newcommand{\Ext}{\mathrm {Ext}}

\newcommand{\RNum}[1]{\uppercase\expandafter{\romannumeral #1\relax}}

\newtheorem{theorem}{Theorem}[section]
\newtheorem{corollary}[theorem]{Corollary}
\newtheorem{proposition}[theorem]{Proposition}
\newtheorem{lemma}[theorem]{Lemma}

\theoremstyle{definition}

\theoremstyle{remark}
\newtheorem{remark}[theorem]{Remark}

\newcommand{\R}{\mathbb{R}}

\newcommand{\attop}[1]{{\let\textstyle\scriptstyle\let\scriptstyle\scriptscriptstyle\substack{#1}}}
\renewcommand{\atop}[1]{{\let\scriptstyle\textstyle\let\scriptscriptstyle\scriptstyle\substack{#1}}}
\newcommand{\too}[1]{\overset{#1}\to}

\makeatletter
\newcommand{\switchmargin}{
\if@reversemargin
\normalmarginpar
\else
\reversemarginpar
\fi
}
\makeatother
\newcommand{\highlighteva}[1]{\ifmmode{\text{\sethlcolor{Lavender}\hl{$#1$}}}\else{\sethlcolor{Lavender}\hl{#1}}\fi}

\makeatletter
\let\c@equation\c@theorem
\makeatother
\numberwithin{equation}{section}

\begin{document}
\title{The reduced ring of the $RO(C_2)$-graded $C_2$-equivariant stable stems}

\author{Eva Belmont}
\address{Department of Mathematics, University of California San Diego, La Jolla, CA 92093, USA}
\email{ebelmont@ucsd.edu}

\author{Zhouli Xu}
\address{Department of Mathematics, University of California San Diego, La Jolla, CA 92093, USA}
\email{xuzhouli@ucsd.edu}

\author{Shangjie Zhang}
\address{Department of Mathematics, University of California San Diego, La Jolla, CA 92093, USA}
\email{shz046@ucsd.edu}

\maketitle

\begin{abstract}
We describe in terms of generators and relations the ring structure of the $RO(C_2)$-graded $C_2$-equivariant stable stems $\pi_\star^{C_2}$ modulo the ideal of all nilpotent elements.
As a consequence, we also record the ring structure of the homotopy groups of the rational $C_2$-equivariant sphere $\pi_\star^{C_2}(\mathbb{S}_\mathbb{Q})$.
\end{abstract}


\section{Introduction}
Nishida's Theorem states that every positive-degree element in the stable homotopy groups of spheres is nilpotent. Consequently, the non-nilpotent elements in the classical stable stems are concentrated in degree 0, and the reduced ring given by the stable stems modulo its nilpotent elements is isomorphic to the integers.

We want to obtain a similar classification of the non-nilpotent elements in the equivariant context. Namely, given a finite group $G$, recall that the $RO(G)$-graded equivariant stable homotopy groups of spheres, which we denote by $\pi_\star^G$, form a graded ring. Let $\fN_G$ denote its nilradical, the ideal generated by the nilpotent elements. We want to understand the structure of the reduced ring 
\[\pi_{\star}^G/\fN_G.\]

In this paper, we give explicit generators and relations for $\pi_\star^{G}/\fN_G$ in
the case $G=C_2$, the cyclic group of order 2; see Theorem \ref{thm:main-ring}.
The main input is the knowledge that the non-nilpotent elements are concentrated in specific computationally accessible degrees, combined with $C_2$-equivariant Adams spectral sequence 
computations.

All of the main pieces of this calculation can be read off of previous work, but the
result itself does not appear in the literature. Our contribution is to put the pieces
together, handling a number of details along the way such as
the determination of part of the zero-stem for arbitrary weight, and
the passage from 2-completed to integral information.

As Nishida's theorem can be thought of as the starting point of chromatic homotopy theory, our work answers an analogous question in $C_2$-equivariant chromatic homotopy theory. Other parts of equivariant chromatic homotopy theory have been studied, giving some expectation about the complexity of the answer.
For example, Balmer and Sanders \cite{BS} proved that $G$-equivariant thick prime ideals are exactly those pulled back from non-equivariant thick prime ideals via the geometric fixed points functors $\Phi^H$ for subgroups $H\leq G$. More to the point, the equivariant nilpotence theorem \cite{BS, BGH} says that $x\in\pi_\star^G$ is non-nilpotent if and only if $\Phi^H(x)$ is non-nilpotent for some subgroup $H$ of $G$. At $G=C_2$, there are only two subgroups to consider, and we have an assembly problem whose pieces can be described simply, but has a fairly complicated answer (see Theorem \ref{thm:main-ring}).

\subsection{Main theorems}
We describe a $\Z$-additive basis for $\pi_\star^{C_2}/\fN_{C_2}$ in Theorem \ref{thm:n(i)}, which emphasizes the 2-divisibility of $\eta^i$ and makes the module structure over the Burnside ring $\pi^{C_2}_{0,0}$ apparent; then we describe the ring structure of $\pi_\star^{C_2}/\fN_{C_2}$ in Theorem \ref{thm:main-ring}, which is our main purpose. We also include a description of the ring structure of the rationalization $\pi_\star^{C_2}(\mathbb{S}_\mathbb{Q})$ (Corollary \ref{cor:rationalization}). 

We use the grading $(s,w)$ to denote $w$ copies of the sign representation $\sigma$ and
$s-w$ copies of the trivial representation. The quantity $s-w$ is called the
\emph{coweight}. 
Let $\rho\in\pi_{-1,-1}^{C_2}$ denote the Euler class
$S^{0,0}\to S^{1,1}$ of the sign representation. The element $\eta\in \pi_{1,1}^{C_2}$ is defined so
that $\Phi^{C_2}(\eta)=2$ and its underlying non-equivariant map is $\Phi^e(\eta) = -\eta_{cl}$, where $\eta_{cl}\in \pi_1^s$ is the classical Hopf invariant one element.
For more on the notation and
choices of generators, see Section \ref{sec:elements}. 

\begin{theorem}\label{thm:n(i)}
The reduced ring of the $C_2$-equivariant stable stems is concentrated in degrees $(s,w)$ where $s=0$ and $w$ is even, or $s-w=0$. Its underlying group structure is a free abelian group of rank 1 in each of those degrees, except for degree $(0,0)$ in which it has rank 2. The additive generators of $\pi_\star^{C_2}/\fN_{C_2}$ are
\begin{align*}
&\rho^i\text{ in degree $(-i,-i)$ for $i\geq 0$}
\\&\omega_n\text{ for $n\in \Z$ in degree $(0,-2n)$}
\\&{\eta^i\over 2^{n(i)}} \text{ in degree  $(i,i)$ for $i\geq 1$}
\end{align*}
where
$$ n(i) = \begin{cases} 
4t-1 & i = 8t
\\4t & i = 8t+j \text{ for } j = 1,2,3,4
\\4t+1 & i = 8t+5
\\4t+2 & i = 8t+6
\\4t+3 & i = 8t+7.
\end{cases} $$
In particular, for $i\geq 1$, $\eta^i$ is uniquely $2^{n(i)}$-divisible in
$\pi_\star^{C_2}/\fN_{C_2}$, and is not $2^{n(i)+1}$-divisible. Let $[C_2/e] \in A(C_2)\cong \Z[C_2]$ denote the nontrivial
generator. Then $[C_2/e]= \omega_0$ acts by multiplication by 2 on $\omega_n$
for all $n$,
and acts as zero on ${\eta^i\over 2^{n(i)}}$ and $\rho^i$ for $i\geq 1$.
\end{theorem}

\begin{remark}
In the $C_2$-Adams spectral sequence, the elements $\omega_n$ are represented by $\tau^{2n}h_0$ when $n\geq 0$ and
${\gamma\over \tau^{2|n|-1}}$ when $n < 0$.
The Adams spectral sequence representatives for ${\eta^i\over 2^{n(i)}}$ can be
read off of Table \ref{tab:Adams-names}.
Also, note $\eta$ is the element in degree $(1,1)$, as $n(1)=0$.
\end{remark}

\begin{theorem}\label{thm:main-ring}
As a ring, $\pi_\star^{C_2}/\fN_{C_2}$ is generated by
\begin{align*}
 & \rho \text{ in degree $(-1,-1)$}
\\ & \omega_n \text{ for } n \in \Z \text{ in degree $(0, -2n)$}
\\ & {\eta^i\over 2^{n(i)}} \text{ for } i \geq 1, i\equiv 1\text{ or } 7\pmod 8 \text{ in degree $(i,i)$}
\end{align*}
subject to the following additional relations.
\begin{enumerate}
	\item $2 = \omega_0 + \rho\eta$
    \item $\omega_n\cdot \omega_m = 2\cdot \omega_{n+m}$ for $n, m \in \Z$.
    \item 
    If $i\equiv 1\pmod 8$, then
	$$ \eta^3 \cdot {\eta^i\over 2^{n(i)}} = \rho^3 \cdot {\eta^{i+6}\over 2^{m(i+6)}}. $$
	If $i\equiv 7\pmod 8$, then
	$$ \eta\cdot {\eta^i\over 2^{n(i)}} = \rho \cdot {\eta^{i+2}\over 2^{n(i+2)}}. $$
    \item $\displaystyle 0 = \rho \cdot \omega_n = {\eta^i\over 2^{n(i)}}\cdot \omega_n$ for all $n\in \Z$, $i\geq 1$.
\end{enumerate}
\end{theorem}

The next statement is obtained by rationalizing Theorem \ref{thm:main-ring}. The
additive structure of this ring is previously known; see for example \cite{GQ}.

\begin{corollary}\label{cor:rationalization}
Let $\mathbb{S}_{\mathbb{Q}}$ denote the $C_2$-equivariant rational sphere.
As a ring, we have
$$ \pi_\star^{C_2}(\mathbb{S}_\mathbb{Q}) \cong \pi_\star^{C_2}\otimes \Q \cong \Q[\rho, \eta, \omega_1, \omega_{-1}]\big/ \sim, $$ where $$\ |\rho|=(-1, -1), |\eta|=(1, 1), |\omega_1|=(0, -2), |\omega_{-1}|=(0, 2) $$
and $\sim$ denotes the relations
\begin{align*}
\omega_1\cdot \omega_{-1}  & = 4-2\rho\eta 
\\\rho\cdot \omega_{\pm 1}  &= \eta\cdot \omega_{\pm 1} = 0
\\ \rho\cdot (2-\rho\eta) & =\eta\cdot (2-\rho\eta)=0.
\end{align*}
\end{corollary}

\begin{proof}
By Corollary \ref{cor:torsion-free}, the rationalization of $\pi^{C_2}_\star/\fN_{C_2}$ in Theorem \ref{thm:main-ring} is a subring of $\pi_\star^{C_2}(S_\Q)$. Comparing the rank in each degree with Theorem \ref{thm: rank of stable stems}, we see that the inclusion is in fact an equality.
\end{proof}



For the rest of this paper, we will implicitly work over the reduced ring to avoid problems arising from nilpotent elements in the same degree. In particular, one should think of a non-nilpotent element as a coset.

\subsection{Notation.}$\ $\\
$\pi_*^s$: the classical stable homotopy groups of spheres.\\
$\pi_\star^G$: the $RO(G)$-graded stable homotopy groups of spheres.\\
$\pi_{s, w}^{C_2}$: the $RO(C_2)$-graded stable homotopy groups of spheres with $RO(C_2)$-degree $(s-w)+w\sigma$. \\
$\fN_G$: the nilradical in the $RO(G)$-graded stable homotopy groups of spheres. \\
$\fN$: the nilradical $\fN_{C_2}$.\\
$\pi_{s, w}^{C_2}/\fN$: the reduced ring of $RO(C_2)$-graded stable homotopy groups of spheres in degree $(s, w)$.\\
$\pi_{s,w}^\R$: the $\R$-motivic homotopy groups in stem $s$ and motivic weight $w$.\\
$\Ext_{\C}^{s,f,w}$: the $E_2$-page of the $\C$-motivic Adams spectral sequence
in stem $s$, Adams filtration $f$, and motivic weight $w$.\\
$\Ext_{C_2}^{s,f,w}$: the $E_2$-page of the $C_2$-equivariant Adams spectral
sequence in degree $(s,f,w)$.\\
$\Ext_{\R}^{s,f,w}$: the $E_2$-page of the $\R$-motivic Adams spectral sequence
in degree $(s,f,w)$.\\
$\Ext_{NC}^{s,f,w}$: the summand of $\Ext_{C_2}^{s,f,w}$ from \cite[\S2,3]{GHIR}. \\
$E_r^+$: the $\rho$-Bockstein spectral sequence which converges to $\Ext_\R$.\\
$E_r^-$: the $\rho$-Bockstein spectral sequence which converges to $\Ext_{NC}$.\\
$\Phi^{C_2}$: the $C_2$-geometric fixed point functor as in \cite{BS}.\\
$\Phi^e$: the forgetful functor taking a $C_2$-equivariant spectrum to its underlying non-equivariant spectrum.

\subsection{Organization} 
In Section \ref{sec:preliminaries} we review general results about
equivariant nilpotence and define several important elements.
In Section \ref{sec:stem0}, we review the $\R$-motivic and $C_2$-equivariant
Adams spectral sequence and use them to describe $\pi_\star^{C_2}/\fN$ in stem
zero. In Section \ref{sec:cw0}, we describe $\pi_\star^{C_2}/\fN$ in coweight
zero. In Section \ref{sec:proofs}, we assemble previous results to prove the main theorems. In Section \ref{a few rmk} we discuss the $p$-completed ring structure and higher multiplicative structure in the reduced ring. 

\subsection{Acknowledgements}
The first author was supported by NSF Grant DMS-2204357.
The second author was supported by NSF Grant DMS-2105462.

\section{Notation and preliminaries} \label{sec:preliminaries}

\subsection{Prior results on equivariant non-nilpotence}
Next we recall some theorems about $C_2$-equivariant non-nilpotent elements that form the starting point for our work. Iriye first proved an analogue of the Nishida nilpotence theorem in the equivariant context.
\begin{theorem}[\cite{I}] \label{thm:torsion-nilpotent}
Let $G$ be a finite group.
Every torsion element of $\pi_\star^G$ is nilpotent.
\end{theorem}
From now on, we specialize to $G=C_2$ and write $\fN := \fN_{C_2}$.

\begin{theorem}[\cite{Lin}, \cite{GM}, \cite{behrens-shah}, \cite{GQ}]\label{thm: rank of stable stems}
$\pi_{s, w}^{C_2}$ has infinite order if and only if
    \[s=w\quad \ or \quad s=0\ \mathrm{and}\ w\ \mathrm{is}\ \mathrm{even}\]
Moreover, in these cases the rank of $\pi_{s, w}^{C_2}$ as a free abelian group is 1 unless $(s,w)=(0,0)$, in which case the rank is 2.
\end{theorem}

\begin{remark}
Greenlees and Quigley \cite{GQ} studied the analogue of Theorem \ref{thm: rank of stable stems} for general finite groups.
\end{remark}

Balmer and Sanders \cite{BS}, and Barthel, Greenlees and Hausmann \cite{BGH} proved the following equivariant nilpotence theorem.
\begin{theorem}\label{thm: BS, BGH}
An element $\alpha\in \pi_\star^{C_2}$ is nilpotent if and only if both its geometric fixed points $\Phi^{C_2}(\alpha)\in \pi_*^s$ and its underlying non-equivariant map $\Phi^e(\alpha)\in \pi_*^s$ are nilpotent.
\end{theorem}

By Theorems \ref{thm:torsion-nilpotent} and \ref{thm: rank of stable stems} and Corollary \ref{cor:torsion-free} below, $\pi_\star^{C_2}/\fN$ is a free abelian group with ranks bounded above by those specified in Theorem \ref{thm: rank of stable stems}. We will see that this upper bound is achieved in all degrees.

\begin{corollary}\label{cor:torsion-free}
$\pi_\star^{C_2}/\fN$ is torsion-free.
\end{corollary}
\begin{proof}
Suppose $x\in \pi_\star^{C_2}$ is non-nilpotent. By Theorem \ref{thm: BS, BGH}
and the fact (from Nishida's Nilpotence Theorem) that the non-nilpotent elements
of $\pi_*^s$ are concentrated in $\pi_0^s \cong \Z$, there is some
$H\subseteq C_2$ such that $\Phi^H(x)\neq 0$ is in stem 0, hence non-torsion
and non-nilpotent. We need to check that
$nx$ is non-nilpotent for every $n\neq 0$. But $\Phi^H((nx)^k)=n^k \Phi^H(x)^k$
is nonzero for every $n,k$, and hence so is $(nx)^k$.
\end{proof}


\subsection{Some elements in $\pi_\star^{C_2}$}\label{sec:elements}
Recall $RO(C_2)$ is generated by the trivial 1-dimensional representation and the sign representation $\sigma$. In accordance with the grading convention for $\R$-motivic homotopy theory, we use the double grading
\[\pi_{s, w}^{C_2}\]
to indicate the $RO(C_2)$-degree $(s-w)+w\sigma$. We call $s$ the stem, $w$ the weight, and $s-w$ the coweight.

In coweight 0, there is a map
\[\rho: S^{0,0}\to S^{1,1}\]
given by inclusion of fixed points, which is a non-nilpotent element in $\pi_{-1,-1}^{C_2}$. Equivalently, this is the Euler class of the sign representation $\sigma$ of $C_2$, sometimes written as $a_\sigma$.

The non-equivariant Hopf map $\eta_{cl} : S^3\to S^2$ can be modeled as the defining quotient map $\C^2-\{0\}\to \C P^1$ for complex projective space. The quotient map is compatible with the action of complex conjugation so we get a $C_2$-equivariant homotopy class,
and we define $\eta$ in $\pi_{1,1}^{C_2}$ to be the \emph{negative} of this map.
This definition is the negative of that in \cite{GHIR} and most other authors
since it satisfies $\Phi^e(\eta)=-\eta_{cl}$, but we choose it here because it
satisfies
$$ \Phi^{C_2}(\eta)=2 $$
(see \cite[Remark 10.8]{GHIR}).
Since 2 is non-nilpotent, this relation implies that $\eta$ is non-nilpotent.

For a finite group $G$, let $A(G)$ denote the Burnside ring. Segal \cite{S}
proved that there is an isomorphism
\[A(G) \overset{\cong}{\to} \pi^G_{0}.\]
In the case $G=C_2$, we have $\pi^{C_2}_{0,0} = A(C_2)$  with free basis $1$ and $[C_2/e]$, where $1$ denotes the identity and $[C_2/e]$ denotes the generator of $C_2$, which means we have the relation $[C_2/e]^2=2[C_2/e]$. Note that elements of the Burnside ring are all non-nilpotent.

There is also a key relation in $\pi_{0,0}^{C_2}$ (see \cite[Lemma 10.9]{GHIR}):
\begin{equation}\label{eq:epsilon} [C_2/e]= 2 - \rho \cdot \eta. \end{equation}
Moreover, $[C_2/e]$ is represented by $h_0$ in the $C_2$-equivariant
Adams spectral sequence
(see \cite[Definition 11.6]{GHIR}).

\section{Non-nilpotent elements in stem 0} \label{sec:stem0}

By Theorem \ref{thm: rank of stable stems}, there are two cases to consider. In this section, we determine $\pi^{C_2}_{s,w}$ modulo nilpotents when $s=0$ and $w$ is even. In Propositions \ref{prop:00}, \ref{prop:main (a) positive coweight}, and \ref{prop:main (a) negative coweight}
we identify the additive generators in terms of the $C_2$-equivariant Adams spectral sequence. In Theorem \ref{thm:main-ring (a) (b)} we give a multiplicative description.

\subsection{Review of $C_2$-Adams spectral sequences} We first review some facts
about the $C_2$-equivariant and $\R$-motivic Adams spectral sequences
\begin{align*}
\Ext_{C_2}^{s,f,w} \Rightarrow (\pi_{s,w}^{C_2})^\wedge_2
\\\Ext_{\R}^{s,f,w}\Rightarrow (\pi_{s,w}^\R)^\wedge_2
\end{align*}
where $f$ denotes homological degree, $(s,w)$ in the equivariant case is described in Section \ref{sec:elements}, and $(s,w)$ in the $\R$-motivic case is the pair of total degree and motivic weight.
There is a map of spectral sequences (``Betti realization'') from the second
spectral sequence to the first, and the next proposition says that the map of $E_2$ pages
$\Ext_\R \to \Ext_{C_2}$ is a split inclusion.




\begin{proposition}[\cite{GHIR}]\label{E_2 page splitting}
~\begin{enumerate}
    \item There is a splitting
    \[\Ext_{C_2}=\Ext_\R\oplus \Ext_{NC}.\]
    \item There is a $\rho$-Bockstein spectral sequence converging to $\Ext_{C_2}$ such that a Bockstein differential $d_r$ takes a class $x$ of degree $(s,f,w)$ to a class $d_r(x)$ of degree $(s-1, f+1, w)$.
    Under the splitting in part (a), the spectral sequence decomposes as
    \[E_1^+=\Ext_\C[\rho]\Rightarrow \Ext_\R\]
    \[E_1^-\Rightarrow \Ext_{NC}.\] 
\end{enumerate}
\end{proposition}

Guillou, Hill, Isaksen and Ravenel \cite[\S3]{GHIR} explicitly describe the $\F_2$-generators of $E_1^-$:
\begin{proposition} \label{prop:E1minus}
$E_1^-$ is generated over $\F_2$ by elements of the following two types:
\begin{itemize}
    \item Elements of the form $\frac{\gamma}{\rho^a \tau^b}x$ where $0\leq a, 1\leq b$, and $x$ is an element of $\Ext_\C$ that is $\tau$-free and not divisible by $\tau$. If $x$ has degree $(s, f, w)$ in $\Ext_\C$, then $\frac{\gamma}{\rho^a \tau^b}x$ has degree $(s+a,f, w+a+b+1)$.
    \item Elements of the form $\frac{Q}{\rho^a \tau^b}x$ where $0\leq a, 0\leq b\leq k$, and $x$ is an element of $\Ext_\C$ that is $\tau$-torsion and divisible by $\tau^k$ but not $\tau^{k+1}$. If $x$ has degree $(s, f, w)$ in $\Ext_\C$, then $\frac{Q}{\rho^a \tau^b}x$ has degree $(s+a+1,f-1, w+a+b+1)$.
\end{itemize}
\end{proposition}

\begin{remark}
There is an element $h_0$ in degree $(0,1,0)$ in $\Ext_\C$, and it survives the
$\rho$-Bockstein spectral sequence that converges to $\Ext_\R$. Moreover, it
also survives the $\R$-motivic Adams spectral sequence. In particular, it
detects the homotopy element 
in the Burnside ring by convention. Combining facts in Section \ref{sec:elements}, the homotopy element $2\in \pi_{0,0}^{C_2}$ is detected by $h_0+\rho\eta$ in the $E_\infty$ page.
\end{remark}

\subsection{Additive generators in stem 0}
We remark here that the results in this subsection are largely known (see, for example, \cite{BI}, \cite{BI2}, \cite{DI}, \cite{H}), though not all of the exact statements have been written down before. In this subsection, we assume knowledge of the $E_2$-page of both the $\R$-motivic and $C_2$-equivariant Adams spectral sequences and describe the generators $\omega_n$ that appear in Theorem \ref{thm:n(i)}.
Since these spectral sequences converge to 2-completed homotopy groups, results must be translated from the 2-completed to the integral context (Lemma \ref{lem:non-p-divisible}).

We start by proving the following lemma used in the determination of $(\pi^{C_2}_{0,-2n})^\wedge_2$ for $n>0$ using the Adams spectral sequence
in Proposition \ref{prop:main (a) positive coweight}.
\begin{lemma}\label{3.4}
Any class in negative stem in $\Ext_{C_2}$ is $\rho$-free.
\end{lemma}
\begin{proof}
By Proposition \ref{prop:E1minus}, using the fact that $\Ext_\C$ is concentrated in nonnegative stems, we see that $E_1^-$ is concentrated in nonnegative stems. By Proposition \ref{E_2 page splitting} it suffices to study $\Ext_\R$.
By \cite[Proposition 3.1]{BI}, there is an isomorphism $\Ext_\R(S/\rho)\to
\Ext_\C$ from the $E_2$ page of the $\R$-motivic Adams spectral sequence computing $\pi_{*,*}^\R(S/\rho)$, and $\Ext_\C$ is concentrated in nonnegative stems.
Suppose $x\in \Ext_\R$ has stem $s < 0$ and has $\rho \cdot x = 0$. Then in
the long exact sequence
$$ \dots \to \Ext_\R^{s,f-1,w-1}(S/\rho)\to \Ext_\R^{s,f,w}\too{\rho}
\Ext_\R^{s-1,f,w-1} \to \dots $$
the element $x$ in the middle term pulls back to an element $\tilde{x}\in \Ext_\R(S/\rho) =
\Ext_\C$ in stem $s < 0$, a contradiction.
\end{proof}




In Propositions \ref{prop:00}, \ref{prop:main (a) positive coweight}, and \ref{prop:main (a) negative coweight}, we introduce the generators $\omega_n$ for $n\in \Z$. While we identify $\pi_{0,0}^{C_2}/\fN$ directly in Proposition \ref{prop:00}, in the other cases we start by working in the 2-completed context because of our Adams spectral sequence methods.
\begin{proposition}\label{prop:00}
We have $\pi_{0,0}^{C_2}/\fN \cong \Z\{1,\omega_0\}\cong \Z\{ 1,\rho\eta \}$ where $\omega_0$ is detected
by $h_0$ in the $C_2$-equivariant Adams spectral sequence. Moreover, we have $\Phi^e(\omega_0)=2$ and $2 = \omega_0 + \rho\eta$.
\end{proposition}
\begin{proof}
For $n=0$, we saw in Section \ref{sec:elements} that $\pi_{0,0}^{C_2} \cong \Z[C_2]$ has $\Z$-module generators 1 and $[C_2/e]= 2-\rho\eta$, which is detected by $h_0$ in the Adams spectral sequence. As $\Phi^e([C_2/e])=2$ these generators are both non-nilpotent, and so are all of their integer multiples.
Let $\omega_0:= [C_2/e]$; then the above relation gives $2 = \omega_0 + \rho\eta$.
\end{proof}

\begin{proposition}\label{prop:main (a) positive coweight}
For $n > 0$ there are nonzero elements $\omega_n \in (\pi_{0,-2n}^{C_2})^\wedge_2$ detected in the $C_2$-equivariant Adams spectral sequence by $\tau^{2n}h_0$. They are non-2-divisible and satisfy $\Phi^e(\omega_n)=2$.
\end{proposition}
\begin{proof}
Let $n>0$.
In the $\rho$-Bockstein spectral sequence computing $\Ext_\R$ there is a differential  $d_1(\tau) = \rho h_0$ \cite[Proposition 3.2]{DI}, and so there are differentials $d_1(\tau^{2n+1}) = \rho \tau^{2n}h_0$. 
By \cite[Lemma 3.4]{DI}, $\tau^{2n}h_0$ is a nonzero permanent cycle in the $\rho$-Bockstein spectral sequence for every $n> 0$, and hence it represents a nontrivial class in $\Ext_\R^{0,1,-2n}$.
By the splitting in Proposition \ref{E_2 page splitting}, these are nontrivial classes in $\Ext_{C_2}$ as well.

We show that they are permanent cycles in the $C_2$-Adams spectral sequence.
Recall the Adams differential $d_r$ decreases stem by 1, increases filtration by $r$ and preserves weight. Since every element in stem $-1$ is $\rho$-free by Lemma \ref{3.4} and $\tau^{2n}h_0$ is $\rho$-torsion, $\tau^{2n}h_0$ cannot support any differential. Since $\tau^{2n}h_0$ is in filtration 1, it cannot be hit by a differential.

The functor $\Phi^e$ induces a map of Adams spectral sequences from the $C_2$-equivariant Adams spectral sequence to the classical Adams spectral sequence for the sphere, and $\Phi^e(\tau^{2n}h_0) = h_0$ where $h_0$ detects $2\in \pi_0^s$. Choose a homotopy representative $\omega_n\in(\pi_{0,-2n}^{C_2})^\wedge_2$ such that $\Phi^e(\omega_n)=2$. (Note that other choices have image $2u$ for $u\in \Z_2^\times$.)
It remains to check that these classes are non-2-divisible in
$(\pi_\star^{C_2})^\wedge_2$. Note that any homotopy class $x$ such that $2\cdot x\in\pi_\star^{C_2}$ is detected by $\tau^{2n}h_0$ would be detected in filtration 0. On the other hand, we claim that $\Ext_{C_2}^{0,0, <0}=0$. Since $\Ext_\C^{*,0,*} = \F_2[\tau]$, Proposition \ref{prop:E1minus} shows $\Ext_{NC}^{0,0,\leq 0}=0$. To see $\Ext_\R^{0,\leq0,<0}=0$, note $(E_1^+)^{0,0,<0}=\tau\F_2[\tau]$, and these classes all support Bockstein differentials \cite[Proposition 3.2]{DI}. Thus $\omega_n$ is non-2-divisible for degree reasons.
\end{proof}

\begin{remark}
Balderrama, Culver and Quigley (\cite{BCQ} Theorem 7.4.2) proved a stronger result which determines whether $\rho^r\tau^{2^n(4m+1)}h_i$ is a permanent cycle in $\Ext_\R$. This also implies Proposition \ref{prop:main (a) positive coweight}.
\end{remark}

\begin{proposition}\label{prop:main (a) negative coweight}
For $n > 0$ there are nonzero elements $\omega_{-n} \in (\pi_{0,2n}^{C_2})^\wedge_2$ detected in the $C_2$-equivariant Adams spectral sequence by ${\gamma\over \tau^{2n-1}}$. They are non-2-divisible and satisfy $\Phi^e(\omega_{-n})=2$.
\end{proposition}
\begin{proof}
According to Proposition \ref{prop:E1minus}, there are generators
$\frac{\gamma}{\tau^{2n-1}}$ for $n\geq 1$ in degree $(0, 2n)$ in $E_1^-$. 

We claim that each $\frac{\gamma}{\tau^{2n-1}}$ is a permanent cycle in $\Ext_{NC}$, and also a permanent cycle in $\Ext_{C_2}$. Indeed, $E_1^-$ and $\Ext_{NC}$ vanish in the $(-1)$-stem by Proposition \ref{prop:E1minus} and $\Ext_\R$ vanishes in negative coweight: Morel's connectivity theorem \cite{Morel} implies that $\Ext_\C$ vanishes in negative coweight, and hence so does $E_1^+$. Therefore, the elements $\frac{\gamma}{\tau^{2n-1}}$ are permanent cycles in the $\rho$-Bockstein and Adams spectral sequences for degree reasons. 

Moreover, ${\gamma\over \tau^{2n-1}}$ is non-2-divisible because it is in Adams filtration zero. We have $\Phi^e({\gamma\over \tau^{2n-1}})=h_0$ in the $E_\infty$-page according to \cite[Proposition 3.5]{K}; therefore, as in Proposition \ref{prop:main (a) positive coweight} we have $\Phi^e({\gamma\over \tau^{2n-1}}) = 2$, and we may choose homotopy classes $\omega_{-n}$ detected by these elements such that $\Phi^e(\omega_{-n})=2$.
\end{proof}

Now we need to show that the generators discussed in Propositions \ref{prop:main (a) positive coweight} and \ref{prop:main (a) negative coweight} are in the image of the uncompleted homotopy groups.

\begin{lemma}\label{lem:non-p-divisible}
For $n\neq 0$ let $T\subseteq \pi^{C_2}_{0,-2n}$ denote the torsion subgroup.
Then the completion map
$$ \Z\cong \pi_{0,-2n}^{C_2}/T \to (\pi_{0,-2n}^{C_2})^\wedge_2/T \cong \Z_2 $$
sends the generator to $\omega_n$.
\end{lemma}
\begin{proof}
By Theorem \ref{thm: rank of stable stems} and Corollary
\ref{cor:torsion-free}, $\pi^{C_2}_{0,-2n}$ for $n\neq 0$ is a rank-1 abelian
group.

Let $n>0$. 
Consider the following diagrams for any prime $p$, where the vertical maps are completion.
$$ \xymatrix{
(\pi_{0,-2n}^{C_2})/T\ar[d]\ar[r]^-{\Phi^e} & \pi_0^s\ar[d]
\\(\pi_{0,-2n}^{C_2})^\wedge_p/T\ar[r]^-{\Phi^e} & (\pi_0^s)^\wedge_p
} \hspace{30pt}
\xymatrix{
\Z\ar[r]^-\ell\ar[d] & \Z\ar[d]
\\\Z_p\ar[r]^-\ell & \Z_p
}$$
For $p=2$, we know from Proposition \ref{prop:main (a) positive coweight} that
$\Phi^e(\omega_n)=2$ and $\omega_n$ is not 2-divisible, and so $\ell$ can be expressed as $2$ times a 2-adic unit.
For $p$ odd, we have $(\pi^{C_2}_{0,-2n})^\wedge_p \cong \Z_p\{ \tau^{2n} \}$ (see
\cite[before Proposition 7.11]{behrens-shah}) and the underlying map of
$\tau^{2n}$ is 1 up to $p$-adic units. This shows that $\ell$ is
invertible in $\Z_p$.  Putting together the information at all primes, we have
$\ell=2$. Since $\Phi^e(\omega_n)=2$ its image in $\pi^{C_2}_{0,-2n}/T$ is a
generator.

Next we consider the negative coweight case.
The element $\omega_{-1}\in (\pi_{0,2}^{C_2})^\wedge_2$ is detected by an integral class ${\gamma\over \tau}: S^{2,2}\to S^{2,2}/C_2\simeq S^{2,0}$ which forgets to $2\in \pi_0^s$
(see \cite[Definition 3.4]{K}).
Proposition \ref{prop:main (a) negative coweight} says that $\omega_{-1}$
is not divisible by 2 in $(\pi^{C_2}_{0,2})^\wedge_2$, and so it is not divisible by 2
in $\pi^{C_2}_{0,2}$. Thus $\omega_{-1}$ is a generator of
$\pi_{0,2}^{C_2}/T$.

For $n > 1$, we will show that $(\omega_{-1})^n \in \pi^{C_2}_{0, 2n}$ is divisible by (exactly) $2^n$.
Consider the following diagram.
$$ \xymatrix{
(\pi_{0,2n}^{C_2})/T\ar[d]\ar[r]^-{\Phi^e} & \pi_0^s\ar[d]
\\(\pi_{0,2n}^{C_2})^\wedge_2/T\ar[r]^-{\Phi^e} & (\pi_0^s)^\wedge_2
} \hspace{30pt}
\xymatrix{
\Z\ar[r]^-\ell\ar[d] & \Z\ar[d]
\\\Z_2\ar[r]^-\ell & \Z_2
}$$
Since $\Phi^e((\omega_{-1})^n)=2^n$ we
know that $2^n$ is in the image of the top row, which implies $\ell\mid 2^n$.
Moreover, Proposition \ref{prop:main (a) negative coweight} says that $2$ is in
the image of the bottom row, and 1 is not in the image. Thus $\ell=2$, and
$(\omega_{-1})^n$ is $2^n$-divisible in $(\pi^{C_2}_{0,2n})/T$.
\end{proof}

\subsection{Multiplicative relations in stem 0}
\begin{theorem}\label{thm:main-ring (a) (b)}
As a ring, $\bigoplus_{n\in \Z}\pi_{0, 2n}^{C_2}/\fN$ is generated by $1\in \pi_{0,0}^{C_2}$ and $\omega_n\in \pi_{0,-2n}^{C_2}$ for $n \in \Z$
where $\Phi^e(\omega_n)=2 \in \pi_0^s$,
subject to the relations
$$ \omega_n\cdot \omega_m = 2\cdot \omega_{n+m} $$
for all $n,m\in \Z$. Moreover, $\rho \cdot \omega_n = 0$ in $\pi_\star^{C_2}/\fN$ for all $n\in \Z$.
\end{theorem}
\begin{proof}
First we verify the additive statement.
For $n=0$, this is Proposition \ref{prop:00}.
For $n\neq 0$, Theorems \ref{thm:torsion-nilpotent} and \ref{thm: rank of stable stems} 
imply $\pi_{0,-2n}^{C_2} \cong \Z \oplus T_n$ for a torsion group $T_n$, and $\fN
\supseteq T_n$. Since $\omega_n$ is non-nilpotent and torsion-free, we have
$\pi_{0,-2n}^{C_2}/\fN\cong \pi_{0,-2n}^{C_2}/T_n \cong \Z$.
Thus the additive description in these degrees follows from Lemma \ref{lem:non-p-divisible}.

For the multiplicative relations, observe that for degree reasons, $\omega_n
\cdot \omega_m \in \Z\{ \omega_{n+m} \}$, and we have $\omega_n\cdot \omega_m
= 2\omega_{n+m}$ because of the underlying maps
$\Phi^e(\omega_n)=\Phi^e(\omega_m)=2$
from Propositions \ref{prop:00}, \ref{prop:main (a) positive coweight},
and \ref{prop:main (a) negative coweight}.
For $n\neq 0$, we have $\rho\cdot \omega_n = 0$ in $\pi_\star^{C_2}/\fN$ for degree reasons. For $n=0$, the relation $\omega_0\cdot \rho = 0$ can be observed in $\pi_{*,*}^\R$.
\end{proof}





\begin{remark}
Guillou, Hill, Isaksen and Ravenel computed the homotopy of $ko_{C_2}$ and $k\R$ \cite{GHIR}, where the Hurewicz images also detect most of these nonzero product relations.
\end{remark}

\section{Non-nilpotent elements in coweight zero} \label{sec:cw0}

\begin{proposition}\label{prop:2div}
For $i>0$, $\pi_{i,i}^{C_2}/\fN\cong \Z$ is generated by an element $x_i$ satisfying
$2^{n(i)}x_i = \eta^i$ modulo $\fN$, where $n(i)$ is the function in Theorem \ref{thm:n(i)}.
For $i < 0$, $\pi_{i,i}^{C_2}/\fN\cong \Z$ is generated by $\rho^i$.
\end{proposition}
\begin{proof}
Landweber \cite{L} showed that the image of $\Phi^{C_2}:\pi_{i,i}^{C_2}\to \pi_0^s$ for
$i>0$ is $2^{b(i)}\Z$ for a certain function $b:\Z\to \Z$. It is straightforward
to show that 
$$ b(i) = i-n(i). $$
For $i\geq 1$, let $x_i$
denote an element of $\pi_{i,i}^{C_2}$ satisfying $\Phi^{C_2}(x_i) = 2^{b(i)}$.
We have $\Phi^{C_2}(2^{i-b(i)}x_i - \eta^i) = 0$ which implies $2^{i-b(i)}x_i -
\eta^i = 2^{n(i)}x_i - \eta^i$ is nilpotent. Moreover, if $x$ is any other
element such that $2^{n(i)}x -\eta^i$ is nilpotent, then $x = x_i$ in
$\pi_\star^{C_2}/\fN$ since the reduced ring is uniquely divisible. Thus
$$ {\eta^i\over 2^{n(i)}} =: x_i $$
is a well-defined element of $\pi_{i,i}^{C_2}/\fN$.
To see it generates $\pi_{i,i}^{C_2}/\fN$, simply observe that it cannot be
$p$-divisible for any prime $p$ since $\Phi^{C_2}(x_i)/p$ is not in the image of
$\Phi^{C_2}$ by definition of $x_i$.

For the statement about $\pi_{i,i}^{C_2}$ for $i < 0$, observe that $\rho^i$ is
in a non-nilpotent element in this group, and it must be a generator of $\Z$
because $\Phi^{C_2}(\rho^i)=1$. In particular, $\rho^i$ cannot be divisible by
a non-unit in $\Z$.
\end{proof}

\begin{lemma}\label{lem:2eta}
We have $2\eta = \rho \eta^2$ and $2\rho = \rho^2 \eta$ in $\pi_\star^{C_2}$.
\end{lemma}
\begin{proof}
As $\rho$, 2, and $\eta$ are all in the image of $\pi_{*,*}^\R\to
\pi_{*,*}^{C_2}$, the relations follow from the corresponding $\R$-motivic
relations.
\end{proof}

\begin{lemma}\label{lem:2-to-eta}
For $0\leq j\leq i$,
we have that $\eta^i$ is divisible by $2^j$ in $\pi_\star^{C_2}/\fN$ if and only if
$\eta^{i-j}$ is divisible by $\rho^j$ in $\pi_\star^{C_2}/\fN$.
\end{lemma}
\begin{proof}
First we show that $\pi_\star^{C_2}/\fN$ is uniquely $\eta$-divisible and
$\rho$-divisible in stems $\neq 0$; i.e., for homogeneous elements $x,y\in
\pi_{s,w}^{C_2}/\fN$ with $s\neq 0$,
if $\eta x = \eta y$ then $x = y$, and similarly for $\rho$.
If $\rho x = \rho y$, then $x-y$ is $\rho$-torsion, which implies
$\Phi^{C_2}(x-y)=0$. Since the stem is nonzero, $\Phi^e(x-y)=0$, and so $x-y$ is
nilpotent by Theorem \ref{thm: BS, BGH}.
If $\eta x = \eta y$, then $\rho^2 \eta x = \rho^2 \eta y$, and using Lemma
\ref{lem:2eta} we have $2\rho x = 2\rho y$. Since $\pi_\star^{C_2}/\fN$ is
torsion-free, we have $\rho x = \rho y$, reducing to the previous case.

Suppose $\eta^i = 2^j x$ modulo $\fN$. Then $\eta^i \rho = 2^j \rho x = (\rho
\eta)^j \rho x$ by Lemma \ref{lem:2eta} and $\eta^{i-j} = \rho^j x$ by $\eta$-
and $\rho$-divisibility. This argument is reversible, establishing the converse.
\end{proof}

The next statement is about the $\Z[\rho]$-module structure of the coweight-zero part of the reduced ring, which is natural from the point view of Adams spectral sequence.

\begin{proposition}\label{prop:nn}
$\bigoplus_{n\in \Z}\pi^{C_2}_{n,n}/\fN$ is generated as a $\Z[\rho]$-module by
$$ {\eta^i\over \rho^{m(i)}}\text{ for }i\geq 1$$
subject to the relations $2\eta = \rho \eta^2$ and $2\rho = \rho^2 \eta$, where
   \[
	  m(i) = \begin{cases} 
	  i-1 & i \equiv 0\text{ or }1 \pmod 4
	  \\i-2 & i\equiv 2\pmod 4
	  \\i-3 & i\equiv 3\pmod 4.
	  \end{cases}
   \]
\end{proposition}
\begin{proof}
The relations are in Lemma \ref{lem:2eta}. 
It is a straightforward computation to use Lemma \ref{lem:2-to-eta} to convert
the 2-divisibility information modulo $\fN$ in Proposition \ref{prop:2div} to
$\eta$-divisibility information modulo $\fN$: for all $i$ one must check that
$m(i-n(i))\geq n(i)$ but $m(i-(n(i)+1)) < n(i)+1$.
\end{proof}

\begin{remark}
Guillou and Isaksen \cite{GI} specify the Adams spectral sequence names for
the elements $\eta^i/\rho^{m(i)}$. We reproduce this information in Table \ref{tab:Adams-names}.
\renewcommand{\arraystretch}{2.0}
\begin{table}[H]
\begin{tabular}{@{}l|l|l@{}}
stem &  name in $\pi^{C_2}_\star/\fN$ & name in $\Ext_{C_2}     $                                                                       \\\hline\hline 
1  & $\eta$                         & $h_1 $                                                                                                    \\\hline

$8k-1$ & $\displaystyle\frac{\eta^{8k-1}}{2^{4k-1}}=  \frac{\eta^{4k}}{\rho^{4k-1}}$ & $\displaystyle\frac{Q}{\rho^{4k-2}}h_1^{4k}$   \\[1ex]\hline
$8k+1$ & $\displaystyle\frac{\eta^{8k+1}}{2^{4k}}=\frac{\eta^{4k+1}}{\rho^{4k}}$                                  & $\displaystyle\frac{Q}{\rho^{4k-1}}h_1^{4k+1}$\\
\end{tabular}
\caption{Adams spectral sequence representatives for coweight 0 elements of $\pi_\star^{C_2}/\fN$}
\label{tab:Adams-names}
\end{table}
\end{remark}

\begin{remark}
One can also prove the results in this section using the Mahowald invariants
$M(2^i)$ for $i\geq 1$, which were calculated in \cite{MR}, as input instead of
Landweber's result. In the formulation given by Bruner and Greenlees
\cite{BG}, the Mahowald invariant describes the $\rho$-divisibility of
$C_2$-equivariant elements after $2$-completion; note that Proposition \ref{prop:nn} is essentially a
result about the $\rho$-divisibility of $\eta^i$ for $i\geq 1$.
\end{remark}

\begin{remark}
The elements in Proposition \ref{prop:nn} comprise the Bredon-Landweber region
of $\pi_\star^{C_2}$ described in \cite{GI}. In that paper, Guillou and Isaksen
provide an alternate proof of the Mahowald invariants $M(2^i)$. They also illustrated this $\rho$-divisibility of $\eta^i$ in \cite[Figure 2]{GI}. 
\end{remark}


\section{Proofs of the main theorems} \label{sec:proofs}

\begin{proof}[Proofs of Theorems \ref{thm:n(i)} and
\ref{thm:main-ring}]
The additive generators in Theorem \ref{thm:n(i)} are given in Propositions \ref{prop:00}, \ref{prop:main (a) positive coweight}, \ref{prop:main (a) negative coweight} and \ref{prop:2div}. By Lemma \ref{lem:2eta} we have
\[(2-\rho\eta)\cdot \eta=[C_2/e]\cdot \eta=0,\quad (2-\rho\eta)\cdot \rho=[C_2/e]\cdot \rho=0\]
which concludes the proof of Theorem \ref{thm:n(i)}.
Theorem \ref{thm:main-ring}(a)(b) are proved as
Proposition \ref{prop:00} and Theorem \ref{thm:main-ring (a) (b)}. Theorem
\ref{thm:main-ring}(c) follows from Lemma \ref{lem:2-to-eta} and Proposition \ref{prop:nn}. More explicitly, Proposition \ref{prop:nn} and the fact that \begin{align*}
m(i+1) & =
\begin{cases} 
	m(i) + 1 &  i\equiv 0\pmod 4
	\\m(i) & i\equiv 1\pmod 4
	\\m(i) & i\equiv 2\pmod 4
	\\m(i) + 3 & i\equiv 3\pmod 4
\end{cases}
\end{align*}
imply the relations that
if $i\equiv 0\pmod 4$, then
	\begin{equation}\label{eq:eta3-mult} \eta\cdot {\eta^i\over \rho^{m(i)}} = \rho \cdot {\eta^{i+1}\over \rho^{m(i+1)}} \end{equation}
    and if $i\equiv 1\pmod 4$, then
	\begin{equation}\label{eq:eta-mult} \eta^3 \cdot {\eta^i\over \rho^{m(i)}} = \rho^3 \cdot {\eta^{i+3}\over \rho^{m(i+3)}}. \end{equation}
Since $\eta = {\eta^1\over
\rho^{m(1)}}$ and $\eta \cdot {\eta^i\over \rho^{m(i)}} = {\eta^{i+1}\over
\rho^{m(i+1)}}$ if $i\equiv 1,2\pmod 4$, we do not need to include ${\eta^i\over \rho^{m(i)}}$ for
$i\equiv 2,3\pmod 4$ as ring generators. So the ring generators are $\eta^i\over \rho^{m(i)}$ for $i\equiv 0,1\pmod 4$, and by Lemma \ref{lem:2-to-eta} they correspond to the elements ${\eta^i\over 2^{n(i)}}$ for $i\equiv 0,1\pmod 7$.
Similar calculations allow us to convert \eqref{eq:eta3-mult} and \eqref{eq:eta-mult} into the relations in Theorem \ref{thm:main-ring}(c).
The relations in Theorem \ref{thm:main-ring}(d) are true for degree
reasons: any nontrivial element of $\pi_\star^{C_2}/\fN$ must have stem zero or
coweight zero.

\end{proof}

\section{A few remarks}\label{a few rmk}
\begin{remark}[Higher multiplicative structure]
The Toda bracket structure on $\pi_\star^{C_2}$ does not in general descend to a Toda bracket structure on $\pi_\star^{C_2}/\fN$. However, we may consider Toda brackets in $\pi_\star^{C_2}(\mathbb{S}_\mathbb{Q})$. For degree reasons, all three-fold Toda brackets in $\pi_\star^{C_2}(\mathbb{S}_\mathbb{Q})$ are 0.

There exists a nontrivial four-fold Toda bracket in $\pi_\star^{C_2}$, 
\[\omega_1\in \langle \omega_0^2, \rho, \omega_0, \rho\rangle. \]
To see this, use the $\rho$-Bockstein differential $d_1(\tau)=\rho h_0$ (where
$h_0$ detects $\omega_0$) to show that $\tau^2 h_0$ is in the Massey product
$\langle h_0^2, \rho, h_0, \rho\rangle$ which has zero indeterminacy for degree
reasons, and then use the Moss convergence Theorem \cite{Moss} to show
$\omega_1\in \langle \omega_0^2, \rho, \omega_0,\rho\rangle$. This Toda bracket
in $\pi_\star^{C_2}$ does not contain zero, but it does contain zero after rationalization,
as $\omega_1 = {1\over 4}\omega_0^2 \cdot \omega_1$ in $\pi_\star^{C_2}(\mathbb{S}_{\mathbb{Q}})$.

For another example, the shuffle
$$  \langle \omega_0^2, \rho, \omega_0, \rho\rangle\cdot \omega_{-1} = \omega_0^2\cdot \langle \rho, \omega_0, \rho, \omega_{-1}\rangle $$
shows that $\langle \rho, \omega_0, \rho, \omega_{-1}\rangle$ contains 1 in
$\pi_\star^{C_2}$. After rationalization this bracket equals $\pi_{0,0}^{C_2}(\mathbb{S}_\mathbb{Q})$, so in particular it also contains 0.
\end{remark}

\begin{remark}
We have implicitly calculated the $2$-completed ring structure of $(\pi_\star^{C_2}/\fN)_2^\wedge$. It has the same generators and relations as a $\Z_2$-algebra as in Theorem $\ref{thm:main-ring}$. In particular, the information can be directly read off using the Mahowald invariant of powers of $2$ and the $C_2$-Adams spectral sequence. 

The $p$-completed ring structure $(\pi_\star^{C_2}/\fN)_p^\wedge$ for odd primes $p$ can also be deduced from Theorem \ref{thm:main-ring}.  
\end{remark}

\bibliographystyle{alpha}
\bibliography{bib}

\end{document}